\documentclass[12pt, leqno]{article}
\usepackage{fullpage} 
\usepackage{amsmath}
\usepackage{amssymb}
\usepackage{undertilde}
\usepackage{stmaryrd}
\usepackage{bbding}
\usepackage{enumerate}
\usepackage{graphicx}
\usepackage{xypic}
\xyoption{curve}
\usepackage{multicol}
\xyoption{color}
\usepackage{setspace}
\usepackage{boxedminipage}
\usepackage{nicefrac}
\usepackage{multicol}
\usepackage{amsthm}
  \usepackage[pdftex,
  bookmarks = false,
  pdfstartview = FitBH,
  linktocpage = true,
  pagebackref = true,
  pdfhighlight = /O,
  colorlinks=true,
  linkcolor=blue,
  citecolor=blue,
  filecolor = blue,
  urlcolor = blue,
  menucolor = blue,
]{hyperref}

\numberwithin{equation}{section}
\numberwithin{equation}{section}

\title{The {S}trength of {A}bstraction with {P}redicative {C}omprehension
}

\author{
Sean Walsh\footnote{Department of Logic and Philosophy of Science, 5100 Social Science Plaza, University of California, Irvine, Irvine, CA 92697-5100, U.S.A., swalsh108@gmail.com or walsh108@uci.edu}
}

\date{\today}

\begin{document}

\maketitle

\begin{abstract}
Frege's theorem says that second-order Peano arithmetic is interpretable in Hume's Principle and full impredicative comprehension. Hume's Principle is one example of an \emph{abstraction principle}, while another paradigmatic example is Basic Law~V from Frege's \emph{Grundgesetze}. In this paper we study the strength of abstraction principles in the presence of predicative restrictions on the comprehension schema, and in particular we study a predicative Fregean theory which contains~all the abstraction principles whose underlying equivalence relations can be proven to be equivalence relations in a weak background second-order logic. We show that this predicative Fregean theory interprets second-order Peano arithmetic (cf. Theorem~\ref{thm:main}).
\end{abstract}

\small

\tableofcontents

\normalsize

\newtheorem{thm}{Theorem}[section]
\newtheorem{prop}[thm]{Proposition}
\newtheorem{defn}[thm]{Definition}
\newtheorem{claim}[thm]{Claim}

\newpage

\section{Introduction}\label{sec:SP:01}

The main result of this paper is a predicative analogue of Frege's Theorem (cf. Theorem~\ref{thm:main}). Roughly, Frege's theorem says that one can recover all of second-order Peano arithmetic using only the resources of Hume's Principle and second-order logic. This result was adumbrated in Frege's \emph{Grundlagen} of 1884 (\cite{Frege1884aa}, \cite{Frege1980}) and the contemporary interest in this result is due to Wright's 1983 book \emph{Frege's Conception of Numbers as Objects} (\cite{Wright1983}). For more on the history of this theorem, see the careful discussion and references in Heck \cite{Heck2011aa} pp. 4-6 and Beth~\cite{Beth1959ab}.

More formally, Frege's theorem says that second-order Peano arithmetic is interpretable in second-order logic plus the following axiom, wherein the cardinality operator~$\#$ is a type-lowering function from second-order entities to first-order entities:
\begin{equation}\label{eqn:HP} 
\mbox{\emph{Hume's Principle}}: \forall \; X, Y \; (\#X=\#Y \leftrightarrow \exists \; \mbox{bijection } f:X\rightarrow Y)
\end{equation}
Of course, one theory is said to be \emph{interpretable} in another when the primitives of the interpreted theory can be defined in terms of the resources of the interpreting theory so that the translations of theorems of the interpreted theory are theorems of the interpreting theory (cf. \cite{Walsh2014aa} \S{2} or \cite{Lindstrom2003aa} pp.~96-97 or \cite{Hajek1998} pp. 148-149 or \cite{Visser2006ab} \S 2.2). For a proof of Frege's Theorem, see Chapter 4 of Wright's book (\cite{Wright1983}) or \S{2.2} pp. 1688~ff of \cite{Walsh2012aa}.

The second-order logic used in the traditional proof of Frege's Theorem crucially includes impredicative instances of the comprehension schema. Intuitively, the comprehension schema says that every formula~$\varphi(x)$ in one free first-order variable determines a second-order entity:
\begin{equation}\label{eqn:fullcomp}
\exists \; F \; \forall \;x \; (  Fx \leftrightarrow \varphi(x) )
\end{equation}
The traditional proof of Frege's Theorem uses instances of this comprehension schema in which some of the formulas in question contain higher-order quantifiers (cf. \cite{Walsh2012aa} p. 1690 equations~(44)-(45)). However, there is a long tradition of \emph{predicative mathematics}, in which one attempts to ascertain how much one can accomplish without directly appealing to such instances of the comprehension schema. This was the perspective of Weyl's great book \emph{Das Kontinuum} (\cite{Weyl1918}) and has been further developed in the work of Feferman (\cite{Feferman1964aa}, \cite{Feferman2005ab}). Many of us today learn and know of this tradition due to its close relation to the system~${\tt ACA}_0$ of Friedman and Simpson's project of reverse mathematics (\cite{Friedman1975aa}, \cite{Simpson2009aa}).

However, outside of the inherent interest in predicative mathematics, considerations related to  Frege's philosophy of mathematics likewise suggest adopting the predicative perspective. For, Wright and Hale  (\cite{Hale2001}, cf. \cite{Cook2007aa}) have emphasized that Hume's Principle~(\ref{eqn:HP}) is a special instance of the following:
\begin{equation}\label{eqn:AE}
\mbox{\emph{A[E]}}: \hspace{5mm} \forall \; X,Y \; (\partial_E(X)=\partial_E(Y) \leftrightarrow E(X,Y))
\end{equation}
wherein~$E(X,Y)$ is a formula of second-order logic and~$\partial_E$ is a type-lowering operator taking second-order entities and returning first-order entities. These principles were called \emph{abstraction principles} by Wright and Hale, who pointed out that the following crucial fifth axiom of Frege's \emph{Grundgesetze} of 1893 and 1903 (\cite{Frege1893}, \cite{Frege2013aa}) was also an abstraction principle:
\begin{equation}\label{eqn:BLV}
\mbox{\emph{Basic Law~V}}: \hspace{5mm} \forall \; X,Y \; (\partial(X)=\partial(Y) \leftrightarrow X=Y)
\end{equation}
The operator~$\partial$ as governed by Basic Law~V is called the \emph{extension} operator and the first-order entities in its range are called \emph{extensions}. Regrettably, there is no standard notation for the extension operator, and so some authors write~$\S{X}$ in lieu of~$\partial(X)$. In what follows, the symbol~$\partial$ without any subscripts will be reserved for the extension operator, whereas the subscripted symbols~$\partial_E$ will serve as the notation for the type-lowering operators present in arbitrary abstraction principles~(\ref{eqn:AE}).

While the Russell paradox shows that Basic Law~V is inconsistent with the full comprehension schema~(\ref{eqn:fullcomp}) (cf. \cite{Walsh2012aa} p. 1682), nevertheless Basic Law~V is consistent with predicative restrictions, as was shown by Parsons (\cite{Parsons1987a}), Heck (\cite{Heck1996}), and Ferreira-Wehmeier (\cite{Ferreira2002aa}). This thus suggests the project of understanding whether there is a version of Frege's theorem centered around the consistent predicative fragments of the \emph{Grundgesetze}. This project has been pursued in the last decades by many authors such as Heck (\cite{Heck1996}), Ganea (\cite{Ganea2007}), and Visser (\cite{Visser2009ac}). Their results concerned the restriction of the comprehension schema~(\ref{eqn:fullcomp}) to the case where no higher-order quantifiers are permitted. One result from this body of work says that Basic Law~V~(\ref{eqn:BLV}) coupled with this restriction on the comprehension schema is mutually interpretable with Robinson's~$Q$. Roughly, Robinson's~$Q$ is the fragment of first-order Peano arithmetic obtained by removing all the induction axioms. (For a precise definition of Robinson's~$Q$, see \cite{Hajek1998} p. 28, \cite{Simpson2009aa} p. 4, \cite{Walsh2012aa} p. 1680, \cite{Walsh2014aa} p. 106). Additional work by Visser allows for further rounds of comprehension and results in systems mutually interpretable with Robinson's~$Q$ plus iterations of the consistency statement for this theory, which are likewise known to be interpretable in other weak arithmetics   (\cite{Visser2009ac} p. 147).  In his 2005 book (\cite{Burgess2005}), Burgess surveys these kinds of developments, and writes:
\begin{quote}
[\ldots] I believe that no one working in the area seriously expects to get very much further in the sequence~$Q_m$ while working in predicative Fregean theories of whatever kind (\cite{Burgess2005} p. 145).
\end{quote}
Here~$Q_m$ is the expansion of Robinson's~$Q$ by finitely many primitive recursive function symbols and their defining equations along with induction for bounded formulas (\cite{Burgess2005} pp. 60-63), so that Burgess records the prediction that predicative Fregean theories will be interpretable in weak arithmetics.

The main result of this paper suggests that this prediction was wrong, and that predicative Fregean theories can interpret strong theories of arithmetic (cf. Theorem~\ref{thm:main}). While we turn presently to developing the definitions needed to precisely state this result, let us say by way of anticipation that part of the idea is to work both with (i) an expanded notion of a ``Fregean theory,'' so that it includes several abstraction principles, such as Basic Law~V, in addition to Hume's Principle, and (ii) an expanded notion of ``predicativity,'' in which one allows some controlled instances of higher-order quantifiers within the comprehension schema~(\ref{eqn:fullcomp}). Hence, of course, it might be that Burgess and others had merely conjectured that predicative Fregean theories in a more limited sense were comparatively weak.

This paper is part of a series of three papers, the other two being \cite{Walsh2014ac} and \cite{Walsh2015ab}. These papers collectively constitute a sequel to our paper \cite{Walsh2012aa}, particularly as it concerns the methods and components related to Basic Law~V. In that earlier paper, we showed that Hume's Principle~(\ref{eqn:HP}) with predicative comprehension did not interpret second-order Peano arithmetic with predicative comprehension (cf. \cite{Walsh2012aa} p. 1704). Hence at the outset of that paper, we said that ``in this specific sense there is no predicative version of Frege's Theorem'' (\cite{Walsh2012aa} p. 1679). The main result of this present paper (cf. Theorem~\ref{thm:main}) is that when we enlarge the theory to a more inclusive class of abstraction principles containing Basic Law~V, we do in fact succeed in recovering arithmetic. 

This paper depends on \cite{Walsh2014ac} only in that the consistency of the predicative Fregean theory which we study here was established in that earlier paper (cf. discussion at close of next section). In the paper \cite{Walsh2015ab}, we focus on embedding the system of the \emph{Grundgesetze} into a system of intensional logic. The alternative perspective of \cite{Walsh2015ab} then suggests viewing the consistent fragments of the \emph{Grundgesetze} as a species of intensional logic, as opposed to an instance of an abstraction principle.

This paper is organized as follows. In \S\ref{sec:SP:02} we set out the definitions of the predicative Fregean theory. In \S\ref{sec:interpretsetnegh} it is shown how this predicative Fregean theory can recover full second-order Peano arithmetic. In \S\ref{sec:fragile} it is noted that some theories which are conceptually proximate to the predicative Fregean theory are nonetheless inconsistent. 

\section{Defining a theory of abstraction with predicative comprehension}\label{sec:SP:02}

The predicative Fregean theory with which we work in this paper is developed within the framework of second-order logic. The language $L_0$ of the background second-order logic is an~$\omega$-sorted system with sorts for first-order entities, unary second-order entities, binary second-order entities etc. Further, following the Fregean tradition, the first-order entities are called \emph{objects}, the unary second-order entities are called \emph{concepts}, and the~$n$-ary second-order entities for~$n\geq 1$ are called {\it~$n$-ary concepts}. Rather than introduce any primitive notation for the different sorts, we rather employ the convention of using distinctive variables for each sort: objects are written with lower-case Roman letters~$x,y,z,a,b,c\ldots$, concepts are written with upper-case Roman letters~~$X,Y,Z,A,B,C,F,G,H,U,\ldots$,~$n$-ary concepts for~$n>1$ are written with the upper case Roman letters~$R,S,T$, and~$n$-ary concepts are written with the Roman letters~$f,g,h$ when they are graphs of functions.

Besides the sorts, the other basic primitive of the signature of the background second-order logic $L_0$ are the predication relations. One writes~$Xa$ to indicate that object~$a$ has property or concept~$X$. Likewise, there are predication relations for~$n$-ary concepts, which we write as~$R(a_1, \ldots, a_n)$. The final element of the signature $L_0$ of the background second-order logic are the projection symbols. The basic idea is that one wants, primitive in the signature~$L_0$, a way to move from the binary concept~$R$ and the object~$a$ to its projection~$R[a]=\{b: R(a,b)\}$. We assume that the signature $L_0$ of the background second-order logic is equipped with symbols~$(R,a_1, \ldots, a_m)\mapsto R[a_1, \ldots, a_m]$ from~$(m\mbox{+}n)$-ary concepts~$R$ and an~$m$-tuple of objects~$(a_1,\ldots, a_m)$ to an~$n$-ary concept~$R[a_1, \ldots, a_m]=\{(b_1, \ldots, b_n): R(a_1, \ldots, a_m, b_1, \ldots, b_n)\}$. Further, typically in what follows we avail ourselves of the tuple notation $\overline{a}= a_1, \ldots, a_n$ and thus write predication and projection more succinctly as $R(\overline{a})$ and $R[\overline{a}]$, respectively.

All this in place, we can then formally define the signature $L_0$ of the background second-order logic as follows:
\begin{defn}\label{defnL0}
The signature $L_0$ of the background second-order logic is a many-sorted signature which contains (i) a sort for objects and for each $n\geq 1$ a sort for $n$-ary concepts, (ii) for each $n\geq 1$, an $(n+1)$-ary predication relation symbol $R(a_1, \ldots, a_n)$ which holds between an $n$-ary concept $R$ and an $n$-tuple of objects $a_1, \ldots, a_n$, and (iii) for each $n,m\geq 1$, an $(m+1)$-ary projection function symbol $(R,a_1, \ldots, a_m)\mapsto R[a_1, \ldots, a_m]$ from an $(m+n)$-ary concept~$R$ and an~$m$-tuple of objects~$(a_1,\ldots, a_m)$ to an~$n$-ary concept~$R[a_1, \ldots, a_m]$.
\end{defn}
\noindent As is usual in many-sorted signatures, we adopt the convention that each sort has its own identity symbol, so that technically cross-sortal identities are not well-formed. But we continue to write all identities with the usual symbol~``$=$'' for the ease of readability.

The expansions of second-order logic with which we work are designed to handle abstraction principles~(\ref{eqn:AE}). Hence, suppose that $L$ is an expansion of $L_0$. Suppose that~$E(R,S)$ is an $L$-formula with two free~$n$-ary relation variables for some~$n\geq 1$, with all  free variables of $E(R,S)$ explicitly displayed. Then we may expand $L$ to a signature $L[\partial_E]$ which contains a new function symbol~$\partial_E$ which takes~$n$-ary concepts~$R$ and returns the object~$\partial_E(R)$. Then the following axiom, called \emph{the abstraction principle associated to~$E$}, is an $L[\partial_E]$-sentence:
\begin{equation}\label{eqn:AE2}
\mbox{\emph{A[E]}}: \hspace{5mm} \forall \; R,S \; (\partial_E(R)=\partial_E(S) \leftrightarrow E(R,S))
\end{equation}
This generalizes the notion of an abstraction principle~(\ref{eqn:AE}) described in the previous section in that the domain of the operator~$\partial_E$ can be~$n$-ary concepts for any specific~$n\geq 1$. 

This generalization is warranted by several key examples, such as that of ordinals. Let~$R$ be a binary concept and let~$\mathrm{Field}(R)$ be the unary concept $F$ such that $Fx$ iff there is a $y$ such that $Rxy$ or $Ryx$. Then consider the following formula~$E(R,S)$ on binary concepts:
\begin{align}
&  [(\mathrm{Field}(R),R)\models \mathrm{wo}  \vee (\mathrm{Field}(S),S)\models \mathrm{wo}]\rightarrow\label{eqn:BF}  \\
& \hspace{5mm}\exists \mbox{ isomorphism } f: (\mathrm{Field}(R),R)\rightarrow (\mathrm{Field}(S),S) \notag
\end{align}
In this, ``$\mathrm{wo}$'' denotes the natural sentence in the signature of second-order logic which says that a binary concept is a well-order, i.e. a linear order such that every non-empty subconcept of its domain has a least element. It's not too difficult to see that~$E(R,S)$ is an equivalence relation on binary concepts, and that two well-orders will be~$E$-equivalent if and only if they are order-isomorphic. Just as the Russell paradox shows that Basic Law~V~(\ref{eqn:BLV}) is inconsistent with the full comprehension schema, so one can use the Burali-Forti paradox to show that~$A[E]$ for this~$E$ in equation~(\ref{eqn:BF}) is inconsistent with the full comprehension schema (cf. \cite{Hodes1984} p. 138 footnote, \cite{Boolos1998} pp. 214, 311). To handle these abstraction principles we need to adopt restrictions on the comprehension schema, to which we presently turn.

There are three traditional predicative varieties of the comprehension schema: the first-order comprehension schema, the $\Delta^1_1$-comprehension schema, and the $\Sigma^1_1$-choice schema (cf. \cite{Simpson2009aa} {VII.5-6}, \cite{Walsh2012aa} Definition 5 p. 1683). However, to make the comparison with the full comprehension schema~(\ref{eqn:fullcomp}) precise, we should restate it to include not only concepts but~$n$-ary concepts for all~$n\geq 1$ and to indicate its explicit dependence on a signature:
\begin{defn}
Suppose that $L$ is an expansion of $L_0$. Then the \emph{Full Comprehension Schema for $L$-formulas} consists of all axioms of the form \;$\exists \; R \; \forall \; \overline{a} \; (R\overline{a} \leftrightarrow \varphi(\overline{a}))$, wherein~$\varphi(\overline{x})$ is an $L$-formula, perhaps with parameters, and~$\overline{x}$ abbreviates~$(x_1, \ldots, x_n)$ and~$R$ is an~$n$-ary concept variable for~$n\geq 1$ that does not appear free in~$\varphi(\overline{x})$.
\label{eqn:unaryfullcomp:SP}
\end{defn}
\noindent The most restrictive predicative version of the comprehension schema is then the following, where the idea is that \emph{no} higher-order quantifiers are allowed in the formulas:
\begin{defn}
Suppose that $L$ is an expansion of $L_0$.  The \emph{First-Order Comprehension Schema for $L$-formulas} consists of all axioms of the form \;$\exists \; R \; \forall \; \overline{a} \; (R\overline{a} \leftrightarrow \varphi(\overline{a}))$, wherein~$\varphi(\overline{x})$ is an $L$-formula with no second-order quantifiers but perhaps with parameters, and~$\overline{x}$ abbreviates~$(x_1, \ldots, x_n)$ and~$R$ is an~$n$-ary concept variable for~$n\geq 1$ that does not appear free in~$\varphi(\overline{x})$. \label{pred:comp:schema:SP}
\end{defn}

A more liberal version of the comprehension schema is the so-called~$\Delta^1_1$-comprehension schema. A~$\Sigma^1_1$-formula (resp.~$\Pi^1_1$-formula) is one which begins with a block of existential quantifiers (resp. universal quantifiers) over~$n$-ary concepts for various~$n\geq 1$ and which contains no further second-order quantifiers. One then defines:
\begin{defn} \label{delta11comp:PS} Suppose that $L$ is an expansion of $L_0$. Then the {\it~$\Delta^1_1$-Comprehension Schema for $L$-formulas} consists of all axioms of the form
\begin{equation}
(\forall \; \overline{x} \; \varphi(\overline{x})\leftrightarrow \psi(\overline{x}))\rightarrow \exists \; R \; \forall \; \overline{a} \; (R\overline{a} \leftrightarrow \varphi(\overline{a}))
\end{equation}
wherein~$\varphi(\overline{x})$ is a~$\Sigma^1_1$-formula in the signature of $L$ and~$\psi(\overline{x})$ is a~$\Pi^1_1$-formula in the signature of $L$ that may contain parameters, and~$\overline{x}$ abbreviates~$(x_1, \ldots, x_n)$, and~$R$ is an~$n$-ary concept variable for~$n\geq 1$ that does not appear free in~$\varphi(\overline{x})$ or $\psi(\overline{x})$.
\end{defn}
\noindent Finally, traditionally one also includes amongst the predicative systems the following choice principle:
\begin{defn}
Suppose that $L$ is an expansion of $L_0$.  The {\it$\Sigma^1_1$-Choice Schema for $L$-formulas} consists of all axioms of the form
\begin{equation}
[\forall \; \overline{x} \; \exists \; R^{\prime} \; \varphi(R^{\prime},\overline{x})]\rightarrow \exists \; R \; [\forall \; \overline{x} \; \varphi(R[\overline{x}],\overline{x})]
\end{equation}
wherein the $L$-formula~$\varphi(R^{\prime},\overline{x})$ is~$\Sigma^1_1$, perhaps with parameters, and~$\overline{x}$ abbreviates~$(x_1, \ldots, x_m)$ and~$R$ is an~$(m+n)$-ary concept variable for~$n,m\geq 1$ that does not appear free in~$\varphi(R^{\prime}, \overline{x})$ where~$R^{\prime}$ is an~$n$-ary concept variable.
\label{sigam11choice:PS} \end{defn}
\noindent The~$\Sigma^1_1$-Choice Schema and the First-Order Comprehension Schema together imply the~$\Delta^1_1$-Comprehension Schema (cf. \cite{Simpson2009aa} Theorem V.8.3 pp. 205-206, \cite{Walsh2012aa} Proposition 6 p. 1683). Hence, even if one's primary interest is in the latter schema, typically theories are axiomatized with the two former schemas since they are deductively stronger, and that is how we proceed in this paper.

To the signature $L_0$ of the weak background second-order logic, we want to associate a certain weak background $L_0$-theory. Some of the axioms of this background theory axiomatize the behavior of the predication symbols and the projection symbols. For each $m\geq 1$, one has the following \emph{extensionality axiom}, wherein $R, S$ are $m$-ary concept variables and $\overline{a}=a_1, \ldots, a_m$ are object variables:  
\begin{equation}\label{axiom:ext}
\forall \; R, S \; [R=S \leftrightarrow (\forall \overline{a} \; (R(\overline{a})\leftrightarrow S(\overline{a})))]
\end{equation}
But it should be noted that some authors don't explicitly include the identity symbol for concepts or higher-order entities and simply take it as an abbreviation for coextensionality (cf. \cite{Simpson2009aa} pp. 2-3, \cite{Burgess2005} pp. 14-15). Second, for each $n,m\geq 1$, one has the following \emph{projection axioms} governing the behavior of the projection symbols, wherein $R$ is an $(m\mbox{+}n)$-ary concept variable and $\overline{a}=a_1, \ldots, a_m, \overline{b}=b_1, \ldots, b_n$ are object variables:
\begin{equation}\label{axiom:proj}
\forall \; R \; \forall \; \overline{a}, \overline{b} \; [(R[\overline{a}])(\overline{b}) \leftrightarrow R(\overline{a}, \overline{b})]
\end{equation}
Finally, with all this in place, we can define the weak background theory of second-order logic:
\begin{defn}
The \emph{weak background theory of second-order logic ${\tt \Sigma^1_1\mbox{-}OS}$} is $L_0$-theory consisting of (i) the extensionality axioms~(\ref{axiom:ext}) and the projection axioms~(\ref{axiom:proj}) and (ii) the~$\Sigma^1_1$-Choice Schema for $L_0$-formulas~(Definition~\ref{sigam11choice:PS})  and (iii) the First-Order Comprehension Schema for $L_0$-formulas (Definition~\ref{pred:comp:schema:SP}). \label{eqn:secondorderlogic:PS}
\end{defn}
\noindent In the theory ${\tt \Sigma^1_1\mbox{-}OS}$ and its extensions, we use standard abbreviations for various operations on concepts, for instance $X\cap Y = \{z: Xz \; \& \; Yz\}$ and $\{x\}=\{z: z=x\}$ and $X\times Y =\{(x,y): Xx \; \& \; Yy\}$ and $\emptyset = \{x: x\neq x\}$. In general, we use $\{x: \Phi(x)\}$ as an abbreviation for the concept $F$ such that $Fx$ iff $\Phi(x)$, assuming that $\Phi(x)$ is a formula which falls under one of the comprehension principles available in the theory in which we are working.

This weak background theory ${\tt \Sigma^1_1\mbox{-}OS}$ of second-order logic is used to define the following Fregean theory at issue in this paper. If $E(R,S)$ is an $L_0$-formula with two free $n_E$-ary concept variables and no further free variables, then we let $\mathrm{Equiv}(E)$ abbreviate the $L_0$-sentence expressive of $E$ being an equivalence relation on $n_E$-ary concepts, i.e. the universal closure of the following, wherein $R,S,T$ are $n_E$-ary concept variables:
\begin{equation}\label{eqn:mymyequiv}
[E(R,R) \; \& \; (E(R,S)\rightarrow E(S,R)) \; \& \; ((E(R,S) \; \& \; E(S,T))\rightarrow E(R,T))]
\end{equation}
Then consider the following collection of $L_0$-formulas which consists of all the $L_0$-formulas $E(R,S)$ with two free $n_E$-ary concept variables and no further free variables such that ${\tt \Sigma^1_1\mbox{-}OS}$ proves $\mathrm{Equiv}(E)$:
\begin{equation}\label{eqn:provL0}
\mathrm{ProvEquiv}(L_0)=\{E(R,S) \mbox{ is an $L_0$ formula}: {\tt \Sigma^1_1\mbox{-}OS}\vdash \mathrm{Equiv}(E)\}
\end{equation}
Then define the following expansion of $L_1$ of $L_0$:
\begin{defn}\label{defn:L1}
Let $L_1$ consist of the expansion of the signature $L_0$~(\ref{defnL0}) by a new function symbol $\partial_E$ from $n_E$-ary concepts to objects for each $E$ from $\mathrm{ProvEquiv}(L_0)$~(\ref{eqn:provL0}). 
\end{defn}
Then we define the predicative theory as follows:
\begin{defn}\label{defn:pft}
The \emph{predicative Fregean theory}, abbreviated ${\tt PFT}$, is the $L_1$-theory consisting of 
(i) the extensionality axioms~(\ref{axiom:ext}) and the projection axioms~(\ref{axiom:proj}) and (ii) the~$\Sigma^1_1$-Choice Schema for $L_1$-formulas~(Definition~\ref{sigam11choice:PS})  and (iii) the First-Order Comprehension Schema for $L_1$-formulas (Definition~\ref{pred:comp:schema:SP}), and (iv) the abstraction principle $A[E]$~(\ref{eqn:AE2}) for each $E$ from $\mathrm{ProvEquiv}(L_0)$~(\ref{eqn:provL0}).
\end{defn}
\noindent Hence, the theory ${\tt PFT}$ is a recursively enumerable theory in a recursively enumerable signature $L_1$. If one desired a recursive signature, one could alternatively define $L_1$ to consist of function symbols $\partial_E$ from $n_E$-ary concepts to objects for \emph{each} $L_0$-formula $E$, regardless of whether it was in $\mathrm{ProvEquiv}(L_0)$~(\ref{eqn:provL0}). This is because clause~(iv) in Definition~\ref{defn:pft} only includes the abstraction principle $A[E]$~(\ref{eqn:AE2}) when the formula $E$ is in fact in the set $\mathrm{ProvEquiv}(L_0)$~(\ref{eqn:provL0}).

While this definition is technically precise, the niceties ought not obscure the intuitiveness of the motivating idea. For, the idea behind this  predicative Fregean theory is that it conjoins traditional predicative constraints on comprehension together with the idea that abstraction principles associated to certain $L_0$-formulae are always available. More capaciously: if we start from weak background theory of second-order logic ${\tt \Sigma^1_1\mbox{-}OS}$ and if we can prove in this theory that an $L_0$-formula~$E(R,S)$ in the signature of this weak background logic is an equivalence relation on~$n_E$-ary concepts for some~$n_E\geq 1$, then the predicative Fregean theory ${\tt PFT}$ includes the abstraction principle~$A[E]$~(\ref{eqn:AE2}) associated to~$E$. Hence the theory ${\tt PFT}$ includes the abstraction principles associated to number, extension, and ordinal, namely Hume's Principle~(\ref{eqn:HP}), Basic Law~V~(\ref{eqn:BLV}) and the abstraction principle associated to ordinals~(cf. (\ref{eqn:BF}) above).

One of the aims of the earlier paper \cite{Walsh2014ac} was to establish the following:
\begin{thm}\label{thm:itsconsistent}
The theory ${\tt PFT}$ is consistent.
\end{thm}
\begin{proof}
Let $E_1, \ldots, E_n, \ldots$ enumerate the elements of the collection $\mathrm{ProvEquiv}(L_0)$ from equation~(\ref{eqn:provL0}). By compactness, it suffices to establish, for each $n\geq 1$, the consistency of the subsystem of ${\tt PFT}$ which is formed by restricting part~(iv) of the Definition of ${\tt PFT}$ to the abstraction principles $A[E_1], \ldots, A[E_n]$. But then this theory is a subtheory of the theory which, in the paper \cite{Walsh2014ac}, we called ${\tt \Sigma^1_1-}{\tt [E_1, \ldots, E_n]A}+{\tt SO}+{\tt GC}$. The consistency of this theory was established in the Joint Consistency Theorem of that paper.
\end{proof}

\section{Interpreting second-order arithmetic in the theory}\label{sec:interpretsetnegh}

While the predicative Fregean Theory only explicitly includes predicative instances of the comprehension schema for $L_0$-formulas, surprisingly it is able to deductively recover all instances of the Full Comprehension Schema for $L_0$-formulas.

\begin{thm}\label{eqn:Iampreviousprop}
${\tt PFT}$ proves each instance of the Full Comprehension Schema for $L_0$-formulas.
\end{thm}
\begin{proof}
Let $\Phi(x,G)$ be an $L_0$-formula with all free variables displayed, wherein $x$ is an object variable and $G$ is a unary concept variable. Let us first show that ${\tt PFT}$ proves the following instance of the Full Comprehension Schema for $L_0$-formulas (Definition~\ref{eqn:unaryfullcomp:SP}):
\begin{equation}\label{eqn:instancecomp}
\forall \; G \; \exists \; F \; \forall \; x \; (Fx \leftrightarrow \Phi(x,G))
\end{equation}
After we finish the proof of this instance, we'll comment on how to establish the general case.

First consider the following $L_0$-formulas $\mu(R,S), \nu(R,S)$ with all free variables displayed, where $R,S$ are binary concept variables:
 \begin{align}
\mu(R,S) & \equiv  [\exists\; ! \; x, G \; \mbox{with } R=\{x\}\times G] \; \& \; [\exists \; ! \; y, H\; \mbox{with } S=\{y\}\times H]\notag \\
  \; \& \; & \forall \; x,G,y,H \; [(R=\{x\}\times G \; \& \; S=\{y\}\times H) \rightarrow (\Phi(x,G)\leftrightarrow \Phi(y,H))] \label{eqn:moo} \notag \\
\nu(R,S) & \equiv  \neg [\exists\; ! \; x, G \; \mbox{with } R=\{x\}\times G] \; \& \; \neg [\exists \; ! \; y, H\; \mbox{with } S=\{y\}\times H] \notag
 \end{align}
 In this, the identity $R=\{x\}\times G$ is an abbreviation for the claim that 
 \begin{equation}
\forall \; a,b\; (R(a,b)\leftrightarrow ((a=x) \; \& \; Gb))
 \end{equation}
 Hence, $\mu(R,S)$ expresses that $R$ can be written uniquely as $\{x\}\times G$ for some $x,G$, while $S$ can be written uniquely as $\{y\}\times H$ for some $y,H$, and that $\Phi(x,G)\leftrightarrow \Phi(y,H)$. The circumstance in which a binary relation $R$ can be written as $\{x\}\times G$ but not \emph{uniquely} so is when $G$ is empty, since in this case $\{x\}\times G = \{x^{\prime}\}\times G$ for any objects $x,x^{\prime}$. Finally, consider the following $L_0$-formula $E(R,S)$ where again $R,S$ are binary concept variables and all free variables are displayed:
\begin{equation}\label{eqn:defnEEEE}
E(R,S)\equiv (\mu(R,S) \vee \nu(R,S)) 
\end{equation}

The weak background theory ${\tt \Sigma^1_1\mbox{-}OS}$ proves that $E(R,S)$ is an equivalence relation on binary concepts. For reflexivity, either $R$ can be written uniquely as $\{x\}\times G$ for some $x,G$, or not. If so, then one trivially has $\Phi(x,G)\leftrightarrow \Phi(x,G)$. This then implies $\mu(R,R)$ and so $E(R,R)$. If not, then of course $\nu(R,R)$ and so $E(R,R)$. For symmetry, it simply suffices to note that both $\mu$ and $\nu$ are symmetric in that $\mu(R,S)$ implies $\mu(S,R)$ and likewise for $\nu$. For transitivity, suppose that $E(R,S)$ and $E(S,T)$. Because of the disjunctive definition of~$E$ in (\ref{eqn:defnEEEE}), there are three cases to consider. First suppose that $\mu(R,S)$ and $\mu(S,T)$. Then we may uniquely write $R=\{x\}\times G, S=\{y\}\times H, T=\{z\}\times I$, and from $\Phi(x,G)\leftrightarrow \Phi(y,H)$ and $\Phi(y,H)\leftrightarrow \Phi(z,I)$ we may conclude that $\Phi(x,G)\leftrightarrow \Phi(z,I)$. Hence we then have $\mu(R,T)$ and thus $E(R,T)$. Second suppose that $\nu(R,S)$ and $\nu(S,T)$. These two assumptions imply that we can't write any of $R, S,T$ uniquely as the product of a singleton and a unary concept, and hence that $\nu(R,T)$ and $E(R,T)$. Finally, suppose that $\mu(R,S)$ and $\nu(S,T)$ (or vice-versa). But this case leads to a contradiction, since $\mu(R,S)$ implies that we can write $S$ uniquely as the product of a singleton and a unary concept, while $\nu(S,T)$ says that we can't. Hence $E(R,S)$ is indeed an equivalence relation on binary concepts, and provably so in the weak background theory ${\tt \Sigma^1_1\mbox{-}OS}$.

Then the $L_0$-formula $E(R,S)$ is in the set $\mathrm{ProvEquiv}(L_0)$~(\ref{eqn:provL0}). Hence the theory ${\tt PFT}$ contains the abstraction principle~$A[E]$~(\ref{eqn:AE2}). Before we verify~(\ref{eqn:instancecomp}), let us introduce another abstraction principle. Consider the following $L_0$-formulas $\mu^{\prime}(X,Y), \nu^{\prime}(X,Y)$ with all free variables displayed, where $X,Y$ are unary concept variables:
 \begin{align}
\mu^{\prime}(X,Y) & \equiv  \exists \; x \; \exists \; y \; X=\{x\} \; \& \; Y=\{y\} \; \& \; (\Phi(x,\emptyset)\leftrightarrow \Phi(y,\emptyset)) \notag \label{eqn:moo2} \\
\nu^{\prime}(X,Y) & \equiv  \neg (\exists \; x \; X=\{x\}) \; \& \; \neg (\exists \; y \; Y=\{y\}) \notag
 \end{align}
Then consider the following $L_0$-formula $E^{\prime}(X,Y)$ where again $X,Y$ are unary concept variables and all free variables are displayed:
\begin{equation}\label{eqn:defnEEEE2}
E^{\prime}(X,Y)\equiv (\mu^{\prime}(X,Y) \vee \nu^{\prime}(X,Y)) 
\end{equation}
By the same argument as the previous paragraph, ${\tt \Sigma^1_1\mbox{-}OS}$ proves that $E^{\prime}(X,Y)$ is an equivalence relation unary concepts. So the theory ${\tt PFT}$ contains the abstraction principle~$A[E^{\prime}]$~(\ref{eqn:AE2})

Now, working in ${\tt PFT}$, let us verify~(\ref{eqn:instancecomp}). There are three cases. First suppose that there is no $x_0$ with $\Phi(x_0,G)$. Then to establish~(\ref{eqn:instancecomp}) one can take $F=\emptyset$.

As a second case, suppose that there is a $x_0$ with $\Phi(x_0,G)$ and that $G$ is non-empty. Then observe that the graph of the function $f(x)=\partial_E(\{x\}\times G)$ has both a $\Sigma^1_1$- and a $\Pi^1_1$-definition:
\begin{align}
f(x)=y & \leftrightarrow \exists \; R \; (\forall \; a,b \; R(a,b)\leftrightarrow (a=x \; \& \; Gb)) \; \& \; \partial_E(R)=y\notag \\
& \leftrightarrow \forall \; R \; (\forall \; a,b \; R(a,b)\leftrightarrow (a=x \; \& \; Gb))\rightarrow  \partial_E(R)=y\label{eqn:graphme}
\end{align}
These are equivalent because we can use the First-Order Comprehension Schema for $L_1$-formulas to secure that the binary relation $R=\{x\}\times G$ exists. Hence by the $\Delta^1_1$-Comprehension Schema for $L_1$-formulas, the equivalence in (\ref{eqn:graphme}) implies that the graph of $f$ exists as a binary concept. Then by First-Order Comprehension Schema for $L_1$-formulas, the following unary concept exists:
\begin{equation}
F = \{x: f(x) = \partial_E(\{x_0\}\times G)\}
\end{equation}
Now let's argue that $F = \{x: \Phi(x,G)\}$. First suppose that $Fx$. Then $f(x) = \partial_E(\{x_0\}\times G)$ and hence $\partial_E(\{x\}\times G)=\partial_E(\{x_0\}\times G)$. Then $E(\{x\}\times G, \{x_0\}\times G)$ and since $G$ is non-empty we have $\mu(\{x\}\times G, \{x_0\}\times G)$. Then $\Phi(x, G)\leftrightarrow \Phi(x_0, G)$. Since we're assuming that $\Phi(x_0,G)$, we then conclude that $\Phi(x, G)$, which is what we wanted to show. For the converse, suppose that $\Phi(x, G)$. Since we're assuming that $\Phi(x_0,G)$ and that $G$ is non-empty we may conclude that $\mu(\{x\}\times G, \{x_0\}\times G)$ and thus $E(\{x\}\times G, \{x_0\}\times G)$ and $\partial_E(\{x\}\times G)=\partial_E(\{x_0\}\times G)$. By the definition of $f$, we then have $f(x)=\partial_E(\{x_0\}\times G)$ which by the definition of $F$ implies that $Fx$, which is what we wanted to show.

As a third case, suppose that there is an $x_0$ with $\Phi(x_0,G)$ but that $G$ itself is empty. Then we argue as before that the graph of $g(x)=\partial_{E^{\prime}}(\{x\})$ exists as a binary concept, that $F=\{x: g(x)=\partial_{E^{\prime}}(\{x_0\})\}$ exists as a unary concept, and that $F=\{x: \Phi(x,G)\}$.

This finishes the proof of (\ref{eqn:instancecomp}) in ${\tt PFT}$. The proof of the general case of the Full Comprehension Schema for $L_0$-formulas (Definition~\ref{eqn:unaryfullcomp:SP}) differs only in that unary concept variable~$F$ from~(\ref{eqn:instancecomp}) might instead be an $n$-ary concept variable and there may be more than one concept parameter~$G$, as well as some additional object parameters. But the proof of this general case is directly analogous to the proof of (\ref{eqn:instancecomp}). The only difference is that the number of abstraction principles used in the proof will increase with the number of concept parameters. In general if there are $m$-concept parameters $G_1, \ldots, G_m$, then there will be $2^m$ different abstraction principles used in the proof, since one must consider a case corresponding to the finite binary sequence $(i_1, \ldots, i_m)$, wherein $i_k=0$ indicates that $G_k$ is empty, and $i_k=1$ indicates that $G_k$ is non-empty.
\end{proof}

Before turning to the proof that ${\tt PFT}$ interprets second-order Peano arithmetic, let's briefly note that in the consistency proof from \cite{Walsh2014ac} invoked in the proof of Theorem~\ref{thm:itsconsistent}, we explicitly verified the Full Comprehension Schema for $L_0$-formulas. (In the language of that paper, these were part of the theory~${\tt SO}$, and the interested reader may consult the proof of the Joint Consistency Theorem in that paper).

While the theory ${\tt PFT}$ only explicitly includes some instances of the Full Comprehension Schema for $L_0$-formulas in its definition~(cf. Definition~\ref{defn:pft}), the previous theorem says that it proves all of them. However, even in this predicative setting, the Russell paradox can be used to show that there is no concept consisting of the extensions, i.e. the range of the extension operator~$\partial$ from Basic Law~V (\ref{eqn:BLV}). For a proof, see \cite{Walsh2012aa} Proposition 29 p. 1692. Now the formula $\mathrm{rng}(\partial)$ is definable by a $\Sigma^1_1$-formula of the signature $L_0[\partial]$. Further $L_0[\partial]$ is included in the signature~$L_1$ of ${\tt PFT}$. Hence, since the $L_1$-theory ${\tt PFT}$ is consistent by Theorem~\ref{thm:itsconsistent}, it follows that  ${\tt PFT}$ does not prove all instances of the Full Comprehension Schema for $L_1$-formulas.

This kind of situation is of course not entirely unfamiliar. For instance, Presburger arithmetic yields a complete axiomatization of the structure $(\mathbb{Z},0,1+,<)$ (cf. Marker~\cite{Marker2002aa} pp. 82 ff). So this axiomatization proves each instance of the following induction schema in the signature $L=\{0,1,+,<\}$:
\begin{equation}\label{eqn:scheamex}
[\varphi(0) \; \& \; \forall \; x\geq 0 \; (\varphi(x)\rightarrow \varphi(x+1)))]\rightarrow [\forall \; x\geq 0 \; \varphi(0)]
\end{equation}
Consider a non-standard model $G=(G,0,1,+,<)$ of Presburger arithmetic, and extend $L$ to $L^{\prime}$ by adding a new unary predicate $Z$ which is interpreted on $G$ as the integers~$\mathbb{Z}$. Then of course the axioms of Presburger arithmetic do not imply all instances of the schema~(\ref{eqn:scheamex}) in the expanded signature $L^{\prime}$. So of course it's consistent for there to be a schema and an $L^{\prime}$-theory and a subsignature $L$ of $L^{\prime}$ such that the theory proves all instances of the $L$-schema but not every instance of the $L^{\prime}$-schema.

Now let's show that ${\tt PFT}$ interprets second-order Peano arithmetic ${\tt PA}^2$. These axioms are the natural set of axioms used to describe the standard model of second-order arithmetic; see \cite{Simpson2009aa} p. 4 or \cite{Walsh2012aa} p. 1680 or \cite{Walsh2014aa} p. 106 for an explicit list of these axioms.
\begin{thm}\label{thm:main}
The predicative Fregean theory ${\tt PFT}$ interprets second-order Peano arithmetic ${\tt PA}^2$.
\end{thm}
\begin{proof}
First note that the predicative Fregean theory ${\tt PFT}$ proves the existence of the graph of the function $s(x)=\partial(\{x\})$ (cf. \cite{Walsh2012aa} Proposition 27 p. 1691), where this is the abstraction operator associated to Basic Law~V~(\ref{eqn:BLV}). For, note that in ${\tt PFT}$, for all objects $x,y$, one has that the following $\Sigma^1_1$-condition and $\Pi^1_1$-conditions are equivalent:
\begin{equation}\label{eqn:provequiv}
[\exists \; X \; (X=\{x\} \; \& \; \partial(X)=y)] \leftrightarrow [\forall \; X \; (X=\{x\} \rightarrow \partial X = y)]
\end{equation}
By the $\Delta^1_1$-Comprehension Schema for $L_1$-formulas, there is then a binary relation which holds of objects $x,y$ iff either the $\Sigma^1_1$-condition holds or the $\Pi^1_1$-condition  holds. And this binary relation is obviously the graph of the function $s(x)=\partial(\{x\})$.

Let $M$ be $\{x: x=x\}$, which exists by Full Comprehension for $L_0$-formulas, and let $0=\partial(\emptyset)$. Then one has that the triple $(M,0,s)$ satisfies the first two axioms of Robinson's~$Q$:
\begin{equation}\label{eqn:BLVcerfiied}
\forall \; x \; s(x)\neq 0, \hspace{10mm} \forall \; x, y \; (s(x)=s(y)\rightarrow x=y) 
\end{equation}
For, suppose that $s(x)=0$. Then $\partial(\{x\})=\partial(\emptyset)$ and then by Basic Law~V~(\ref{eqn:BLV}) one has that $\{x\}=\emptyset$, a contradiction. Similarly, suppose that $s(x)=s(y)$. Then $\partial(\{x\})=\partial(\{y\})$ and so by Basic Law~V~(\ref{eqn:BLV}) one has that $\{x\}=\{y\}$ and hence $x=y$. Thus (\ref{eqn:BLVcerfiied}) follows immediately from Basic Law~V~(\ref{eqn:BLV}).

But then standard arguments allow one to interpret second-order Peano arithmetic ${\tt PA}^2$ by taking the natural numbers $N$ to be the sub-concept of $M$ consisting of all those subconcepts of $M$ which are ``inductive,'' that is which contain zero and closed under successor. Here of course for the existence of $N$ and the verification of the other axioms of arithmetic, one appeals to the Full Comprehension Schema for $L_0$-formulas, using $M,0,s$ as parameters (cf. \cite{Walsh2012aa} Theorem 16 p. 1688).
\end{proof}

\section{The fragility of abstraction with predicative comprehension}\label{sec:fragile}

However, in spite of its technical strength, the conceptual basis of the predicative Fregean theory ${\tt PFT}$ is rather fragile. For, the $L_1$-theory ${\tt PFT}$ was formed by adding the abstraction principle $A[E]$ associated to the $L_0$-formulas $E(R,S)$ when this formula could be proven to be an equivalence relation in the background second-order logic ${\tt \Sigma^1_1\mbox{-}OS}$. But one cannot successively iterate this idea. For, suppose that in analogue to $\mathrm{ProvEquiv}(L_0)$ in equation~(\ref{eqn:provL0}), one defines:
\begin{equation}\label{eqn:provL1}
\mathrm{ProvEquiv}(L_1)=\{E(R,S) \mbox{ is an $L_1$ formula}: {\tt PFT}\vdash \mathrm{Equiv}(E)\}
\end{equation}
And further suppose that one defines $L_2$ to be the expansion of $L_1$ by the addition of a function symbol $\partial_E$ from $n_E$-ary concepts to objects for each $L_1$-formula $E(R,S)$ in $\mathrm{ProvEquiv}(L_1)$. Finally, suppose one defines the following iteration of ${\tt PFT}$ (cf. Definition~\ref{defn:pft}):
\begin{defn}\label{defn:pft2}
The theory ${\tt PFT}_2$ is the $L_2$-theory consisting of 
(i) the extensionality axioms~(\ref{axiom:ext}) and the projection axioms~(\ref{axiom:proj}) and (ii) the~$\Sigma^1_1$-Choice Schema for $L_2$-formulas~(Definition~\ref{sigam11choice:PS})  and (iii) the First-Order Comprehension Schema for $L_2$-formulas (Definition~\ref{pred:comp:schema:SP}), and (iv) the abstraction principle $A[E]$~(\ref{eqn:AE2}) for each $E$ which is from $\mathrm{ProvEquiv}(L_0)$~(\ref{eqn:provL0}) or from $\mathrm{ProvEquiv}(L_1)$~(\ref{eqn:provL1}).
\end{defn}
\noindent Then the same argument as in the proof of Theorem~\ref{eqn:Iampreviousprop} establishes that ${\tt PFT}_2$ proves each instance of the Full Comprehension Schema for $L_1$-formulas. But then ${\tt PFT}_2$ is inconsistent, since on pain of the Russell paradox there is no concept of all extensions (cf. \cite{Walsh2012aa} Proposition 29 p. 1692), where again the extensions are the range of the abstraction operator~$\partial$ associated to Basic Law~V~(\ref{eqn:BLV}). Hence, while the predicative Fregean theory ${\tt PFT}$ is consistent, when one tries to iterate its underlying idea of adding abstraction principles when their equivalence relations can be proven to be equivalence relations, one again runs up against the Russell paradox. This indicates that the resource of abstraction principles in the predicative setting is unlike that of typed~theories of truth or second-order logic, which we may consistently add to any consistent theory.

This point is underscored when one observes that the same considerations show the inconsistency of an axiom-based analogue of the rule-based predicative Fregean theory ${\tt PFT}$. In particular, suppose that we recursively defined a signature $L^{\ast}$ extending $L_0$ so that if $E(R,S)$ is an $L^{\ast}$-formula in exactly two free $n_E$-ary concept variables then $L^{\ast}$ also contains a function symbol $\partial_E$ which takes $n_E$-ary concepts to objects and which does not occur in $E$. One could then define the following $L^{\ast}$-theory:
\begin{defn}\label{defn:pft2}
The theory ${\tt PFT}^{\ast}$ is the $L^{\ast}$-theory consisting of (i) the extensionality axioms~(\ref{axiom:ext}) and the projection axioms~(\ref{axiom:proj}) and (ii) the~$\Sigma^1_1$-Choice Schema for $L^{\ast}$-formulas~(Definition~\ref{sigam11choice:PS})  and (iii) the First-Order Comprehension Schema for $L^{\ast}$-formulas (Definition~\ref{pred:comp:schema:SP}), and (iv) the axiom $\mathrm{Equiv}(E)\rightarrow A[E]$ for each  $L^{\ast}$-formula $E$.
\end{defn}
\noindent In this, $\mathrm{Equiv}(E)$ is the sentence which says that $E$ is an equivalence relation (cf. (\ref{eqn:mymyequiv})) and $A[E]$ is the abstraction principle~(\ref{eqn:AE2}), so that the axiom $\mathrm{Equiv}(E)\rightarrow A[E]$ says that if $E$ is an equivalence relation, then $A[E]$ holds. The considerations of the previous paragraphs can be replicated in this theory ${\tt PFT}^{\ast}$, showing it to be inconsistent. However, the conceptual distance between the inconsistent $L^{\ast}$-theory ${\tt PFT}^{\ast}$ and the consistent $L_1$-theory ${\tt PFT}$ is rather slim. The difference is merely a difference between a rule and an axiom: whereas the rule-based ${\tt PFT}$ only includes an abstraction principle when the underlying equivalence relation is expressible in the weak background logic and is \emph{provably} an equivalence relation there, the axiom-based ${\tt PFT}^{\ast}$ includes a commitment to either the truth of the abstraction principle or the falsity of its underlying formula being an equivalence relation.

In response to this, one might try to restrain the predicative Fregean theory ${\tt PFT}$ so that the analogously defined iterated version of it and the analogously defined axiom-based version of it were consistent. For instance, one might consider restricting the abstraction principles added to the theory ${\tt PFT}$ to those whose underlying equivalence relation was expressible both as a $\Sigma^1_1$-formula and a $\Pi^1_1$-formula in the background second-order logic. This, it might be suggested, would be a genuinely predicative theory of abstraction principles. Such a move would block the proof of Theorem~\ref{eqn:Iampreviousprop}. For, the equivalence relation $E(R,S)$~(\ref{eqn:defnEEEE}) used in that proof is not obviously expressible in such a way. However, it is unknown to us how much arithmetic this more austerely predicative theory could interpret, and it is not obvious to us whether the analogously defined iterated version of it (or axiom-based version of it) is consistent. 

Another way forward might be to find some principled way to focus attention on abstraction principles which are somehow more like the paradigmatic Basic Law~V~(\ref{eqn:BLV}) and Hume's Principle~(\ref{eqn:HP}) and the abstraction principle associated to ordinals~(\ref{eqn:BF}), and somehow less like the seemingly ad-hoc abstraction principles constructed in the proof of Theorem~\ref{eqn:Iampreviousprop}. But to do so would be to lose some of the original motivation for focusing on predicative abstraction principles. For, part of the attraction was supposed to be that more abstraction principles became consistent and jointly consistent. And indeed, as the predicative Fregean Theory ${\tt PFT}$ attests, a good deal of joint consistency is available in this setting. Hence in the earlier paper \cite{Walsh2014ac} we said that we had resolved an analogue of the joint consistency problem. But as we have seen in this section, when we try to iterate the underlying idea of abstraction principles in the predicative setting, we again run into inconsistency and seem back in the situation of trying to discern ways to weed out the acceptable from the unacceptable abstraction principles. For an overview of the various candidates for acceptable abstraction principles in the general impredicative setting, see \cite{Linnebo2010ad} or \cite{Cook2012aa}.

Perhaps another way forward might be to give up on the idea of abstraction principles altogether and find principled reasons for studying systems centered around either Basic Law~V~(\ref{eqn:BLV}) itself or Hume's Principle~(\ref{eqn:HP}) itself or the abstraction principle associated to ordinals~(\ref{eqn:BF}) all by itself. With respect to Basic Law~V~(\ref{eqn:BLV}), this is the perspective of \cite{Walsh2015ab}, where the idea is to work within an intensional logic and see the extension operator as selecting a sense for each concept, just like we might select a specific Turing machine index for each computable function. But much remains unknown about the individual abstraction principles at the predicative level. For instance, it is to our knowledge unknown whether Basic Law~V~(\ref{eqn:BLV}) or the abstraction principle associated to ordinals~(\ref{eqn:BF}), equipped with the $\Sigma^1_1$-choice schema and the First-Order Comprehension Schema, interprets the analogous predicative versions of arithmetic (cf. \cite{Walsh2012aa} p. 1707). In this paper, the idea for interpreting arithmetic was to collect together all the predicative abstraction principles so that they could effect the interpretation together, and it is in general unclear to us what happens when one focuses on the abstraction principles one by one.

\subsection*{Acknowledgements}
I was lucky enough to be able to present parts of this work at a number of workshops and conferences, and I would like to thank the participants and organizers of these events for these opportunities. I would like to especially thank the following people for the comments and feedback: Robert~Black, Roy~Cook, Matthew~Davidson, Walter~Dean, Marie~Du\v{z}\'{\i}, Kenny~Easwaran, Fernando~Ferreira, Martin~Fischer, Rohan~French, Salvatore~Florio, Kentaro~Fujimoto, Jeremy~Heis, Joel~David~Hamkins, Volker~Halbach, Ole Thomassen~Hjortland, Luca~Incurvati, Daniel~Isaacson, J\"onne~Kriener, Graham~Leach-Krouse, Hannes~Leitgeb, {\O}ystein~Linnebo, Paolo~Mancosu, Richard~Mendelsohn, Tony~Martin, Yiannis~Moschovakis, John~Mumma, Pavel~Pudl\'ak, Sam~Roberts, Marcus~Rossberg, Tony~Roy, Gil~Sagi, Florian~Steinberger, Iulian~Toader, Gabriel~Uzquiano, Albert~Visser, Kai~Wehmeier, Philip~Welch, Trevor~Wilson, and Martin~Zeman. 

Finally, a special debt is owed to the editors and anonymous referees of this journal, to whom I express my gratitude. For, the proofs were greatly simplified by their suggestions and the previous reliance upon choice was removed by virtue of these suggestions. While composing this paper, I was supported by a Kurt G\"odel Society Research Prize Fellowship and by {\O}ystein Linnebo's European Research Council funded project ``Plurals, Predicates, and Paradox.'' 

\bibliographystyle{plain}
\bibliography{mybib.v011.bib}

\end{document}